\documentclass[a4paper]{article}

\usepackage{a4wide}
\usepackage[utf8]{inputenc}
\usepackage{graphicx}      
\usepackage{cite}
\usepackage{hyperref}

\usepackage{amsmath,amsfonts,amsthm}
\usepackage{braket}

\usepackage{authblk}

\usepackage{mathtools}
\mathtoolsset{showonlyrefs}

\newtheorem{theorem}{Theorem}[section]

\newtheorem{proposition}[theorem]{Proposition}
\newtheorem{lemma}[theorem]{Lemma}

\theoremstyle{definition}
\newtheorem{remark}[theorem]{Remark}
\newtheorem{assumption}[theorem]{Assumption}

\newcommand{\R}{\mathbb{R}}
\newcommand{\N}{\mathbb{N}}
\newcommand{\id}{\mathrm{id}}
\newcommand{\diag}{\mathrm{diag}}
\newcommand{\diff}{\mathrm{d}}
\newcommand{\umax}{u_{\max}}
\newcommand{\Tr}{\mathrm{Tr}}
\newcommand{\E}{\mathbb{E}}
\newcommand{\Var}{\mathrm{Var}}

\newcommand{\eps}{\varepsilon}
\newcommand{\Jac}{\mathrm{Jac}}
\newcommand{\vin}{v_{\ket{1}}}

\newcommand{\omax}{\overline{\omega}}
\newcommand{\omin}{\underline{\omega}}

\newcommand{\kmax}{\overline{\kappa}}
\newcommand{\kmin}{\underline{\kappa}}

\newcommand{\gmax}{\overline{\gamma}}
\newcommand{\gmin}{\underline{\gamma}}

\title{Open qubit parameter identification with bounded pulses}

\author[1]{G. Aloui} 
\author[1]{I. Beschastnyi} 
\author[1]{L. Sacchelli}

\affil[1]{\small Inria, Université Côte d’Azur, CNRS, LJAD, France. 

\footnotesize (email: \texttt{ivan.beschastnyi,ludovic.sacchelli@inria.fr})}

\date{\today}

\begin{document}

\maketitle

\begin{abstract}               
We address the problem of parameter identification for a single open qubit subjected to relaxation and dephasing. Our approach is based on selecting a minimal set of carefully chosen qubit configurations that can be reliably prepared and measured in order to provide an interpretable methodology of parameter identification while potentially minimizing experimental overhead.
The protocol relies on saturating control pulses to generate these configurations. In an idealized regime of infinite-amplitude pulses, we demonstrate that the parameters can be reconstructed analytically from the measured observables. 
We then consider large but finite pulses as a perturbation of this ideal regime and provide bounds on the estimation error introduced by the practical implementation. This framework allows us to separate the sources of uncertainty in the estimation procedure, distinguishing between statistical fluctuations arising from repeated measurements and modeling errors due to deviations from the ideal pulse regime.
\end{abstract}

\section{Introduction}

{Quantum information technologies are emerging as one of the central scientific and technological challenges of the 21st century. A major obstacle to scalable architectures is decoherence, which directly manifests as errors in qubits and limits the reliability of quantum operations \cite{Preskill2018quantumcomputingin}. Addressing this issue requires both error–correction strategies on the algorithmic side and accurate modeling and mitigation of environmental effects on the engineering side \cite{Siddiqi}.
Because decoherence directly translates into operational errors, control strategies that seek to counteract its effects require an accurate description of the underlying system dynamics. 

Consequently, estimating the parameters of controlled quantum Hamiltonians has become a key practical problem in quantum engineering. A variety of approaches, with different levels of generality, have been proposed. Some methods rely on extensive measurement data: for instance, the parameters can be estimated by minimizing a mean-square discrepancy between theoretical and measured expectation values of selected observables, a strategy underlying widely used protocols such as Ramsey interferometry~\cite{ramsey}. Although versatile, such techniques often require a large number of measurements, which increases the experimental and economic cost of calibration. Other approaches focus on constructing estimators that saturate the Quantum Cramér–Rao bound~\cite{geometric_estimation,hamiltonian_identification}. More recently, machine-learning-based methods have also been explored for this task~\cite{quantum_machine}.}

{
In the present work, we carry out a detailed case study on the identification of the parameters of a single qubit subjected to environmental interactions. The system is described by a standard {two-level system subject to relaxation and dephasing.}
Our approach has two complementary objectives. First, we investigate the design of an estimation protocol that relies on a minimal number of experimental observables, to be estimated from repeated projective measurements, while keeping the experimental burden as low as possible. The parameters are inferred from empirical frequencies, and we optimize the free design parameters of the protocol so as to minimize the statistical variance of the resulting estimators. Instead of relying on curve-fitting techniques, we exploit an explicit inversion strategy. While individual parameters, such as decoherence times~\cite{lindblad} or the Larmor frequency~\cite{larmor}, can be estimated by classical techniques, a unified treatment allowing the simultaneous reconstruction of all parameters in a noisy qubit model appears to be missing in the literature. One goal of this work is to fill this gap.

Second, we analyze the role of pulse-based control in enabling parameter identification. Pulses are omnipresent in quantum control, and techniques such as Ramsey interferometry exploit their capacity to isolate specific dynamical effects. In idealized models with unbounded controls, pulses act as instantaneous state rotations, much like discrete gates in quantum computing. While this approximation is structurally valid for pure-state control \cite{Lokutsievskiy_2024}, its quantitative impact on estimation procedures remains less understood. To address this, we design our protocol in an ideal regime of infinitely strong pulses, where the sequence is transparent and analytically tractable, and then treat realistic, finite-amplitude pulses as a perturbation. This allows us to separate the total uncertainty into a statistical component, arising from the measurement noise under the idealized model, and a modeling component due to imperfections of the pulses.

The paper is organized as follows. In Section~\ref{S:setting}, we present the qubit model, the estimation protocol, and the theoretical guarantees we obtain. Section~\ref{S:ideal} develops the protocol in the ideal-pulse regime. Section~\ref{S:bound} quantifies the discrepancy between ideal and finite pulses. Section~\ref{S:local} discusses a refinement based on local disambiguation, and Section~\ref{S:simulations} concludes with numerical simulations.
}

\section{Setting}\label{S:setting}
\subsection{Open qubit dynamics}
The system under study is a two-level open quantum system interacting with its environment (open qubit).
Denoting by $X,Y,Z$ the Pauli matrices, we assume that in the absence of environmental coupling the qubit undergoes coherent dynamics generated by the controlled Hamiltonian
\begin{equation}
H_u=\frac{\omega}{2}Z+\frac{u\kappa}{2}X,
\end{equation} 
where $\omega$ is the intrinsic frequency of the system and $\kappa$ quantifies the strength of the control field. The control field is assumed to take values in a known admissible range $u\in [-\umax,\umax]$ that we assume to be large in comparison to other model parameters.

Environmental interactions are modeled by a relaxation of the qubit towards ground state, of rate $\gamma_1$, and dephasing, of rate $\gamma_2$.

We denote by $\sigma^\pm=(X\pm i Y)/2$ the raising and lowering operators, respectively. 
Then the density matrix $\rho$ of the open qubit evolves according the master equation (see, e.g., \cite{BreuerPetruccione})
\begin{equation}
\dot \rho = -i [H_u,\rho]
+\gamma_1(\sigma^+\rho \sigma^- -\frac{1}{2}\{\sigma^-\sigma^+,\rho\})
+
\gamma_2(Z\rho Z -\rho)
\end{equation}
Adopting the Bloch's ball formalism, we denote by $v=(x,y,z)\in \R^3$ the vector such that $x^2+y^2+z^2\leq 1$ and $\rho=\frac{1}{2}(\id+xX+yY+zZ)$. Under these assumptions, we can rewrite the evolution of $\rho$ as $\dot v= A_u v +b$, where 
$b^\top =\gamma_1\begin{pmatrix}
0&0&1
\end{pmatrix}^\top\in \R^3$ and  $A_u\in \R^{3\times 3}$ can be decomposed as 
\begin{equation}
A_u= \omega\Omega_z+ \kappa u \Omega_x   +\gamma_1 \Gamma_1+\gamma_2 \Gamma_2
\end{equation}
where $\Gamma_1=\diag(-1/2,-1/2,-1)$, $\Gamma_2=\diag(-2,-2,0)$,
\begin{equation}
\Omega_z=
\begin{pmatrix}
    0&-1&0\\1&0&0\\0&0&0
\end{pmatrix},
\qquad
\Omega_x=
\begin{pmatrix}
    0&0&0\\0&0&-1\\0&1&0
\end{pmatrix}.
\end{equation}
In full, we have the dynamics
\begin{equation}\label{E:dynamics}
\begin{pmatrix}
    \dot x\\ \dot y\\ \dot z
\end{pmatrix}
=
\begin{pmatrix}
    -\tfrac{\gamma_1}{2}-2\gamma_2&-\omega&0
    \\
    \omega& -\tfrac{\gamma_1}{2}-2\gamma_2 & -\kappa u
    \\
    0 &\kappa u&-\gamma_1
\end{pmatrix}
\begin{pmatrix}
    x\\y\\z
\end{pmatrix}
+
\begin{pmatrix}
    0 \\ 0 \\ \gamma_1
\end{pmatrix}.
\end{equation}
See Figure~\ref{F:flows} for a decomposed representation of the action of each parameter.
\begin{remark}
One can check that the unique equilibrium at rest ($u=0$) within the unit ball is at the north pole $(0,0,1)$. 
\end{remark}

\begin{figure}
\centering

\begin{minipage}[b]{.3\linewidth}
\centering
\hspace{-3mm}$\omega$

\smallskip
\includegraphics[width=\linewidth]{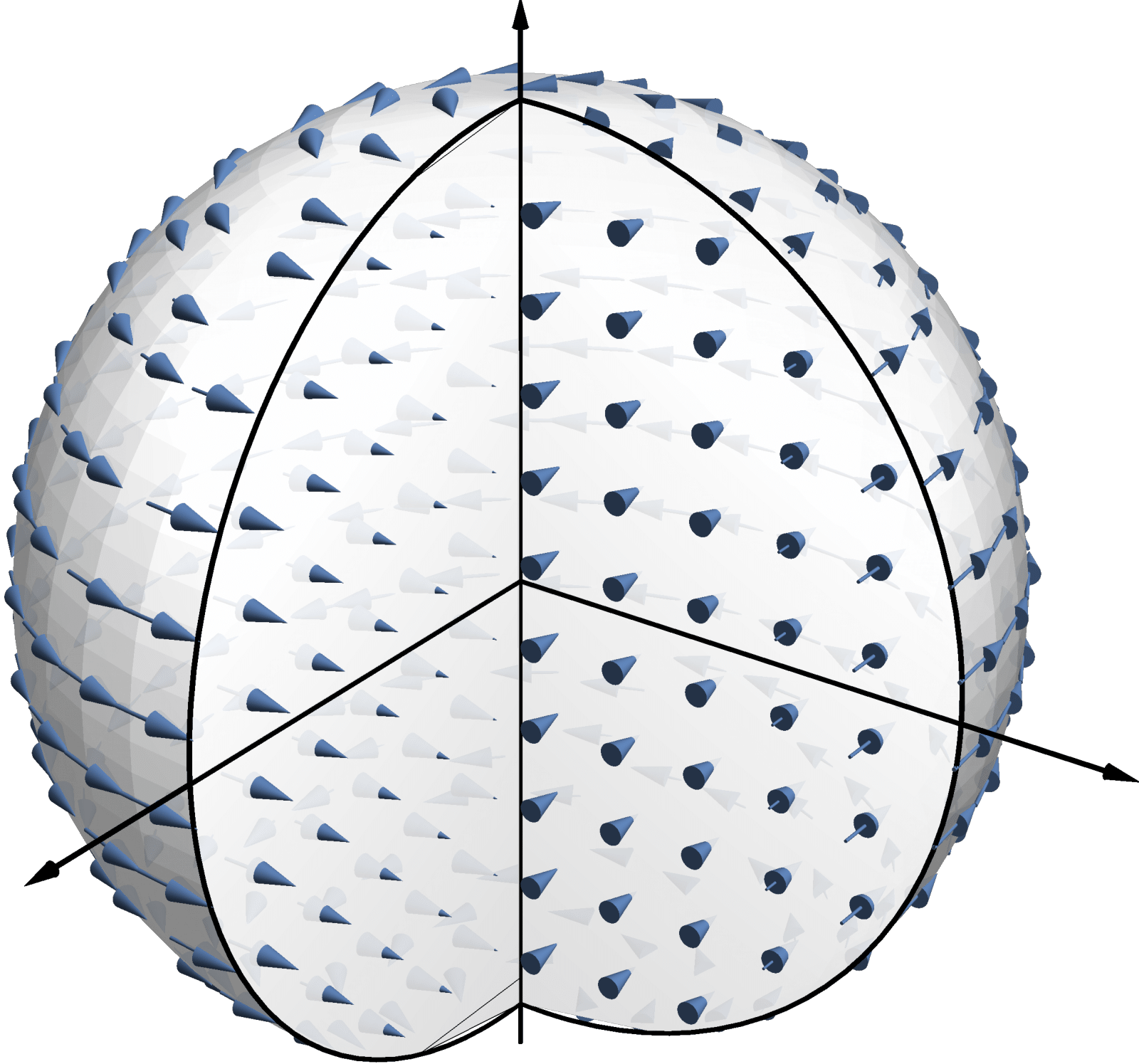} 
\end{minipage}
\hspace{.05\linewidth}
\begin{minipage}[b]{.3\linewidth}
\centering
\hspace{-3mm}$\kappa$

\smallskip
\includegraphics[width=\linewidth]{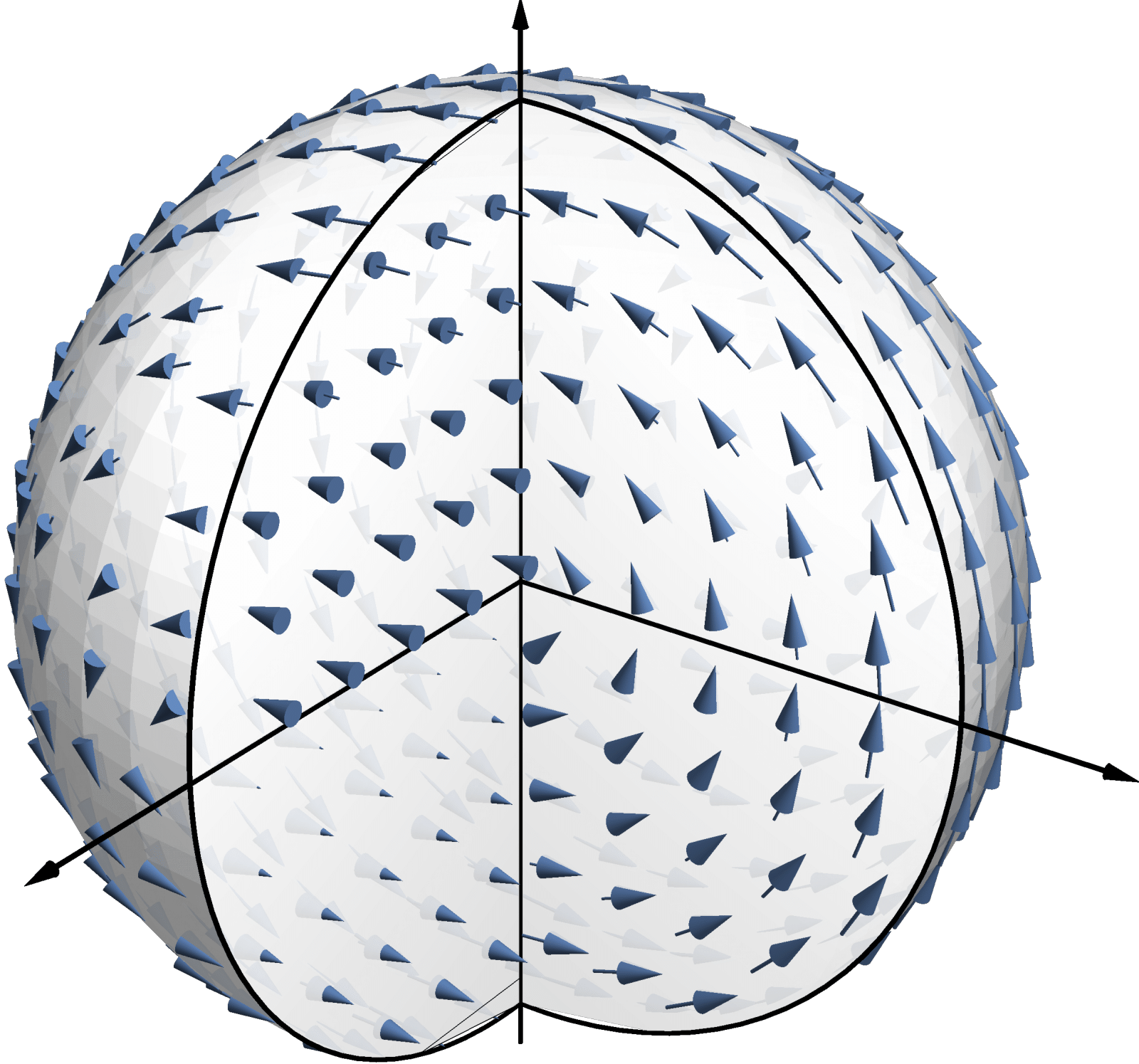} 
\end{minipage}

\bigskip

\begin{minipage}[t]{.31\linewidth}
\centering
\hspace{-3mm}$\gamma_1$

\smallskip
\includegraphics[width=\linewidth]{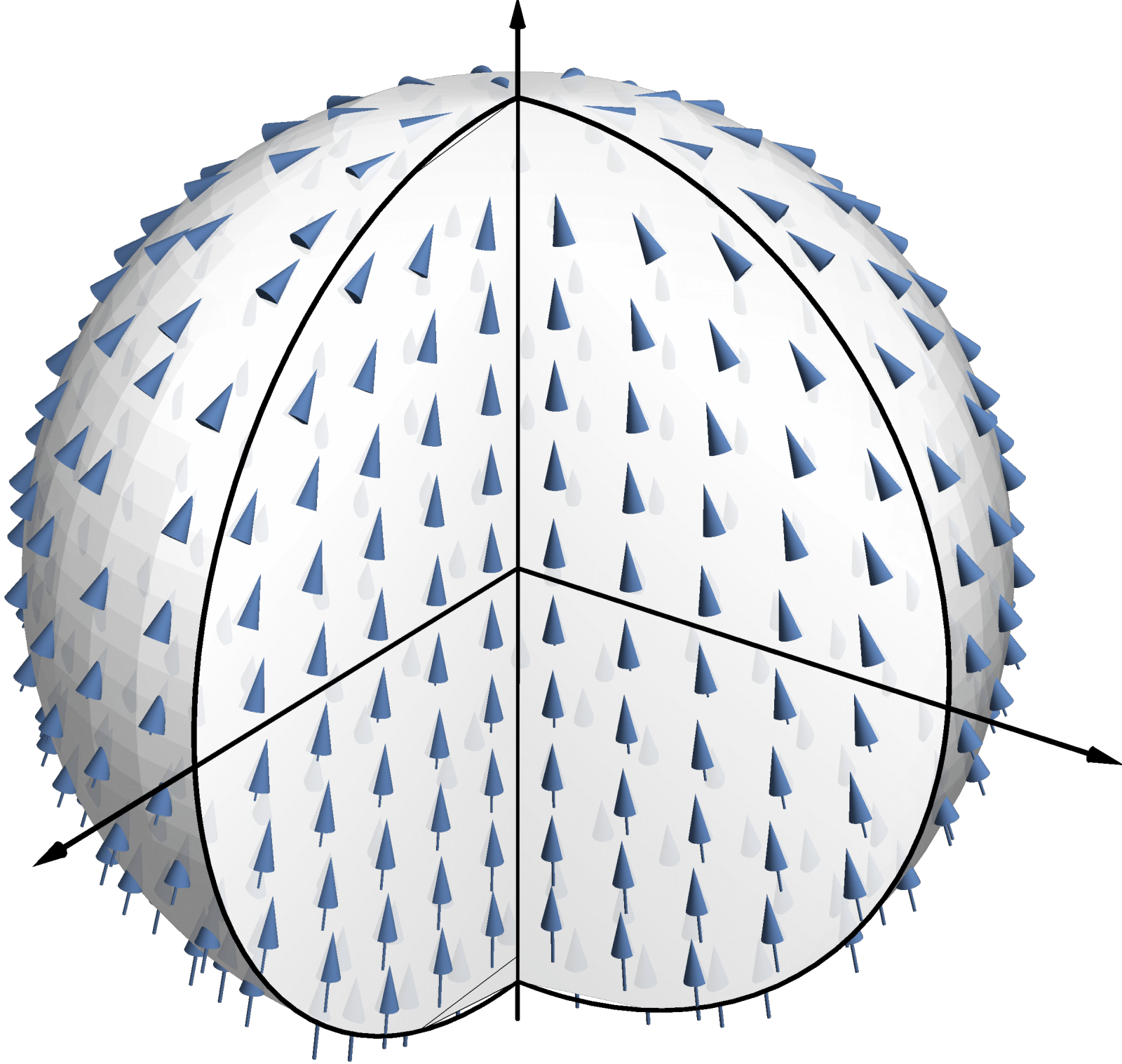} 
\end{minipage}
\hspace{.05\linewidth}
\begin{minipage}[t]{.301\linewidth}
\centering
\hspace{-3mm}$\gamma_2$

\smallskip
\includegraphics[width=\linewidth]{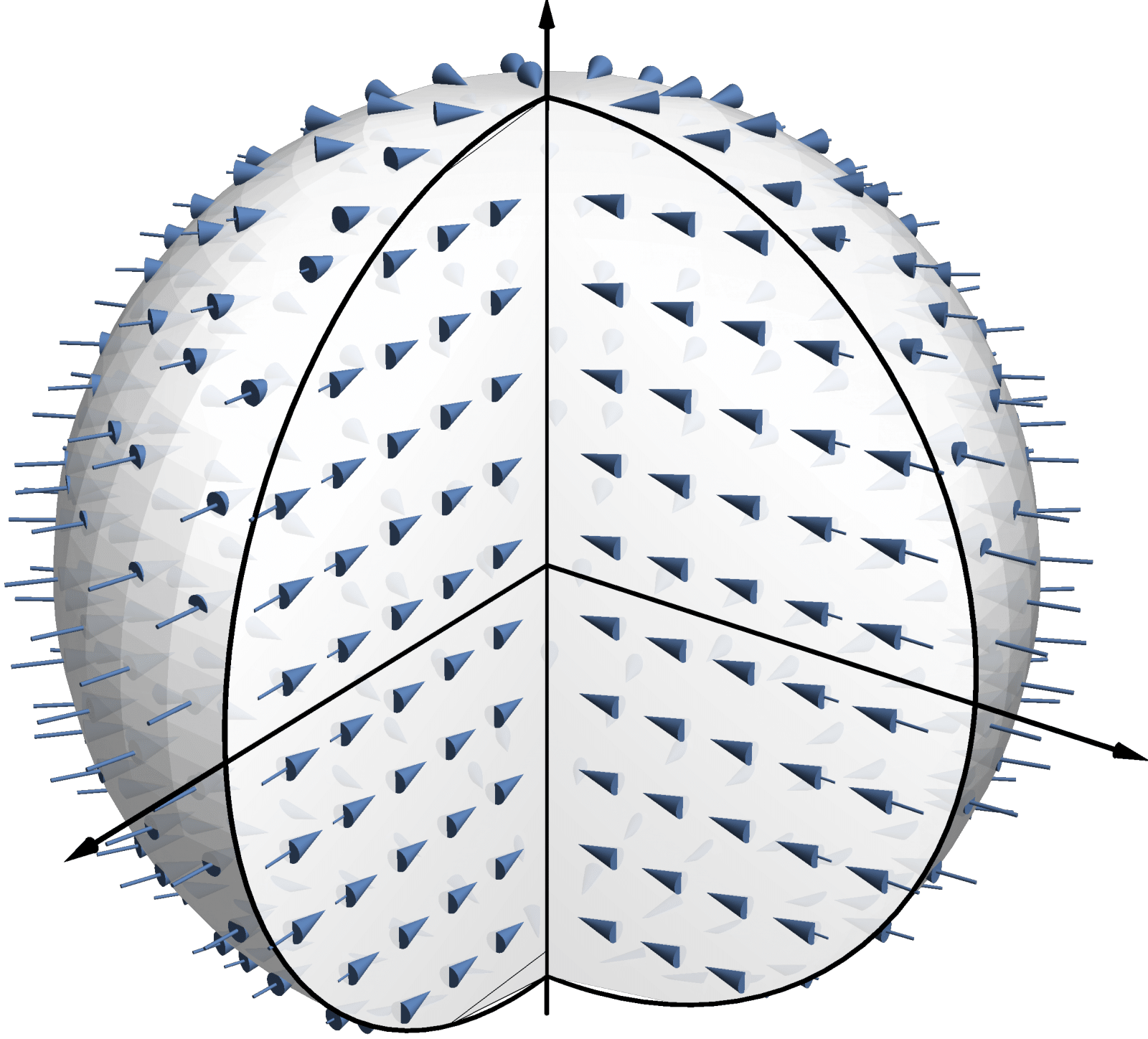} 
\end{minipage}
\begin{minipage}{.8\linewidth}
\caption{Induced flow by each of the parameters taken separately. Notice how the action of $\omega$ and $\gamma_2$ has no component along the $z$-axis.}
\end{minipage}
\label{F:flows}
\end{figure}

\subsection{Estimation protocol}\label{S:estimation}
The system is probed through the projective measurement associated with the observable  $E_1=\ket{1}\bra{1}$, which discriminates between the ground state $\ket{0}\bra{0}$, and the excited state $\ket{1}\bra{1}$. 
In the Bloch-ball representation, these correspond to the north and south poles, $(0,0,1)$ and $(0,0,-1)$ respectively. The probability of obtaining the excited outcome at time $t$ is 
\begin{equation}
    \Tr(E_1\rho(t))=\frac{1-z(t)}{2}.
\end{equation}
Using this projective measurement, we aim to construct an estimation procedure for the unknown system parameters  $\theta:=(\omega,\kappa,\gamma_1,\gamma_2)$ 
by preparing the qubit in preselected configurations by chaining pulses and relaxations and statistically estimating the quantity
 $\Tr(E_1\rho)$. 

\begin{assumption}
Throughout the paper, the initial state $\rho(0)$ is assumed to be the excited state  $\ket{1}\bra{1}$ (i.e. $z=-1$).
\end{assumption}

\begin{remark}
This initialization step can be achieved by repeatedly exciting the system, measuring $E_1$, and conditioning on successful excitations. Sections of the procedure could be adapted to either measurement outcome (with appropriate modifications), but we restrict to the excited-state initialization for clarity of exposition.
\end{remark}

Consider a preparable Bloch state $v\in \bar B(0,1)$.
We assume that this state can be reliably initialized, after which we perform a projective measurement of $E_1$, and record the binary outcomes. Let $\xi^k\in \{0,1\}$ denote the $k$-th measurement result. Then the sequence $(\xi^k)_{k\in \N}$ is an i.i.d. Bernoulli sequence with parameter $p=(1-z)/2=f(\theta)$, where $z$ is the Bloch $z$-coordinate of the prepared state $v$. The sum  $s=\sum_{k=1}^{n} \xi^k$ follows a Binomial law of parameter $(p,n)$, and the estimator $\hat p=s/n$ law well approximated by the Gaussian 
$\mathcal{N}\left(p,\frac{p(1-p)}{n}\right)$, in the sense that $\sqrt{n}(\hat{p}-p)$ converges in law towards $\mathcal{N}\left(0,p(1-p)\right)$.
Therefore estimation of the vector of parameters $\theta$ can be tackled via Delta method (see, e.g., \cite{casella2024statistical}) if we pick the right states $v$ and probabilities $p$.

Estimating $\gamma_1$ is straight-forward with this strategy. Take the initial state 
$\vin=(0,0,-1)$ and let it relax freely during a time $t_1$
The probability of obtaining an excited-state outcome yields the first experimental observable
$$
p_1=\frac{1}{2}(1-z(t_1))=e^{-\gamma_1 t_1}.
$$
Thus $\gamma_1$ can be directly estimated from the empirical observable $\hat p_1$. (See Section~\ref{S:ideal}).

Next we exploit short saturated pulses to create new configurations of the state. When  $u\to \infty$, as $A_u\sim  \kappa u \Omega_x$, the contribution of $\kappa u \Omega_x$ dominates the dynamics on short time frames of order $1/u$. 
In practice, the control amplitude is bounded, the model assuming $u\in [-\umax,\umax]$. We therefore apply a saturated pulse $u=\umax$ for a short duration
$$
t_2 = \frac{\tau_2}{\umax} = \eps \tau_2,\quad \eps:=\frac{1}{\umax},
$$
chosen so that, in the idealised limit $\eps \to 0$, the pulse implements an $x$-rotation of angle $\kappa \tau_2$.

A direct computation (see Section~\ref{S:ideal}), after this pulse, the second observable is
$$
p_2
=
\frac{1}{2}(1+\cos \kappa \tau_2)
+
O(\eps),
$$
where the $O(\eps)$ term quantifies the deviation from the ideal infinite-amplitude pulse. Consequently, the
empirical observable $\hat p_2$ deviates from the ideal $\frac{1}{2}(1+\cos \kappa \tau_2)$, which allows to estimate $\kappa$, with two sources of uncertainty that can be quantified separately: the statistical error (from replacing $p_2$ with $\hat{p}_2$) and the modeling error (from the finite-amplitude approximation).

Finally, let us discuss the case of $\omega$ and $\gamma_2$ without introducing explicit expressions at this stage. Inspecting the dynamics reveals that the contributions of $\gamma_1$ and $\kappa$ contain a component parallel to the measurement axis, whereas the effect of $\omega$ and $\gamma_2$ is essentially orthogonal to it. Consequently, while the observables $p_1$ and $p_2$ were obtained through a pure relaxation and a single short pulse, 
extracting information on $\omega$ and $\gamma_2$ requires creating configurations whose evolution possesses a transverse component. To this end, we introduce two additional observables $p_3,p_4$ by proceeding to a scheme of the form pulse--relaxation--opposite-pulse. This sequence displaces the state away from the measurement axis, allows it to evolve freely so that the transverse dynamics driven by $\omega$ and $\gamma_2$
accumulate, and then returns it near the measurement axis for readout. 

\textbf{Notation.} In order to match the above discussion, we assume now that $\theta=(\gamma_1,\kappa,\gamma_2,\omega)$.

\subsection{Estimator and uncertainty quantification}\label{S:quantification}

Summarizing the previous discussion, we have defined a vector of experimental observables $p\in (0,1)^4$ such that, still as $\eps=1/\umax$, we have a map $F_\eps:\R^4\to \R^4$ such that
$p=F_\eps (\theta).$
We estimate $\theta$ by inverting this relationship: 
\begin{equation}\label{E:inversion}
    \theta = F_\eps^{-1}(p).
\end{equation}
For this inversion to be well posed, $F_\eps$ must be invertible. 

Under suitable assumptions, we show that in the idealized regime $\eps=0$, the mapping is indeed invertible.

\begin{assumption}
The vector of unknown parameters lies in a known box:
\begin{equation*}
(\gamma_1,\kappa,\gamma_2,\omega)
\in 
\Theta=(\gmin_1,\gmax_1)\times (\kmin,\kmax)\times(\gmin_2,\gmax_2)\times(\omin,\omax),
\end{equation*}
with $\bar\Theta$ a compact subset of $(0,\infty)^4$.
\end{assumption}

Under this assumption, with a choice of experimental parameters (pulse and relaxation durations) described in Section~\ref{S:ideal}, we are able to prove (see Lemma~\ref{L:diffeo}) that the mapping $F_0:\Theta\to F_0(\Theta)\subset(0,1)^4$ is a diffeomorphism.

\begin{remark}
The restriction to $\Theta$ guarantees injectivity of $F_0$ by preventing ambiguities in certain observables. For instance, the mapping $\tau_2\mapsto p_2$ cannot be globally injective without restriction. In order to relax this constraint, an adaptive estimation scheme can be employed to disambiguate $\cos \kappa \tau_2$. This is the topic of Section~\ref{S:local}.
\end{remark}

The diffeomorphism property ensures that \eqref{E:inversion} is well posed in the idealized regime. In practice, however, the inverse $F_\eps^{-1}$ is not available in closed form outside of $\eps=0$.
Instead, in Section~\ref{S:bound} we are able to show that bounding pulses by $\umax$ produces a perturbation of order $1/\umax$, quantifiable explicitly (see Lemma~\ref{P:bound_effective}).
\begin{proposition}\label{P:boundInit}
There exists a constant $C>0$ such that 
$$
\sup_{\theta\in \Theta}|F_0(\theta)-F_\eps(\theta)|\leq C \eps.
$$
\end{proposition}
Consequently, we can also write
\begin{equation}\label{E:defG}
\theta = F_0^{-1}(p) + \eps G_\eps(p)
\text{ s.t. }
\sup_{p\in F(\Theta)}\sup_{\eps\in[0,1]} |(G_\eps(p))|<\infty.
\end{equation}
These bounds allows to partially circumvent the issue of inverting $F_\eps$. 
First, since $F_0$ is a diffeomorphism, the perturbative estimate implies that for $\varepsilon$ small enough, $F_\varepsilon$ is a diffeomorphism as well. Indeed, 
$$
\Jac F_\varepsilon^{-1}
=
\Jac F_0^{-1}(\id +\eps  \Jac F_0\Jac G_\varepsilon).
$$
Hence as long as $\eps \|\Jac F_0 \Jac G_\varepsilon\| \leq 1$, the mapping is sure to be a diffeomorphism. Second, from Proposition~\ref{P:boundInit}, we obtain a  computable box $\Delta=[-\Delta_1,\Delta_1]\times \cdots\times [-\Delta_4,\Delta_4]$ such that $F_\eps^{-1}(p)-F_0^{-1}(p)\in \eps \Delta+o(\eps)$.

We now select the estimator
$$
\hat \theta 
=
F_0^{-1}(\hat p)
$$
which is biased due to the $\eps$-perturbation but computationally convenient.

By construction, since measurements generating distinct observables are independent, $\hat p_i$ and $\hat p_j$ are independent if $i\neq j$. As such, we classically have the convergence in law of $\sqrt{n}(\hat{p}-p)$ towards the Gaussian $\mathcal{N}(0, D)$,
letting $D=\diag(p)\cdot(\id -\diag (p))$. Applying the Delta method to approximate the law of $\hat\theta$ yields convergence in law of $\sqrt{n}(\hat{\theta}-mu)$ towards  $\mathcal{N}(0, \Sigma)$, where 
$$
\mu = F_\varepsilon^{-1}(p)= F_0^{-1}(p) =\theta - \eps G_\eps(p)
$$
and 
$$
\Sigma = \Jac F_0^{-1}(p)D\Jac F_0^{-1}(p)^\top.
$$
Thus the uncertainty in $\hat \theta$ can be decomposed cleanly into \textit{(i)} a deterministic modeling bias of order $\eps$ and \textit{(ii)} a statistical uncertainty of order $1/\sqrt{n}$ with modeling bias negligible at this order.

The asymptotic distribution of $\sqrt{n}(\hat{\theta}-\mu)$ allows to find a first order approximation of a $(1-\alpha)100\%$ confidence ellipsoid.
We approximate the covariance $\Sigma$ by substitution with
$$
\hat \Sigma = 
\Jac F_0^{-1}(\hat p)
\diag(\hat p)\cdot(\id -\diag (\hat p))
\Jac F_0^{-1}(\hat p)^\top,
$$
and have the $(1-\alpha)100\%$ confidence ellipsoid for $\mathcal{N}(0, \Sigma)$ approximated by
$$
\hat{\mathcal{E}}_\alpha=
\left\{
\theta\in \R^4
\mid
\theta^\top
\hat\Sigma^{-1}
\theta\leq
\chi^2_{4,1-\alpha}
\right\}
$$
(where $\chi^2_{4,1-\alpha}$ designate the $(1-\alpha)$-quantile of the $\chi^2$ distribution with $4$ degrees of freedom).
From this, we deduce a conclusive main result.

\begin{theorem}\label{T:mainthm}
Under the above assumptions, as $n,\umax\to \infty$, the estimator $\hat \theta$ of $\theta$ is asymptotically unbiased. Moreover, a $(1-\alpha)100\%$ confidence region for $\theta$ is, to first order approximation in $1/\sqrt{n}$ and $1/\umax$,
$$
\mathrm{CR}_{1-\alpha}=\hat{\theta}+\frac{1}{\sqrt{n}}\hat{\mathcal{E}}_\alpha+\frac{1}{\umax}\Delta +o\left(\frac{1}{\sqrt{n}}+\frac{1}{\umax}\right).
$$
\end{theorem}

In the next two sections we focus on the technical aspects that allow to prove this statement.

\section{Estimation in the idealized infinite pulse regime}\label{S:ideal}

In the present section, we discuss the design of the experiment in order to produce the mapping $F$ between the space of parameters and the space of experimental observables. This design is achieved in the idealized asymptotic regime $\umax\to \infty$. Again, letting $\eps = 1/\umax$, and $F_0$ designate the asymptotic ideal mapping, we prove that $F_0$ is a diffeomorphism under the right assumptions.

Finally, throughout the paper, we consider the mappings $h:v\mapsto \frac{1}{2}(1-z)$ and 
$$
\phi(v_0;u,t)
=
e^{A_u t} v_0
+
\int_0^t e^{A_u (t-s)} b \diff s.
$$
As such,  $\phi(v_0;u,t)$ is the state at time $t$ of a solution of System~\ref{E:dynamics} with initial condition $v_0$ and under constant control $u$, and $h(v)$ is the probability $p$ associated with the observable $E_1$ for the state $v$.

In particular, we introduce three main transformation following from $\phi$: letting $t_3$ be a parameter to be set later and $\eps=1/\umax$ as usual:
\begin{equation}
\begin{aligned}
P^+_{t}(v)
&=
\phi(v;\umax,t),
\\
P^-_{t}(v)
&=
\phi(v;-\umax,t),
\\
R_{t}(v)
&=
\phi(v;0,t).
\end{aligned}
\end{equation}
The transformations $P^+$ and $P^{-}$ are two saturated pulses, while $R_t$ is a relaxation. We now use these transformation to design the experimental observables.

\subsection{Parameters parallel to the observation}

\subsubsection[Estimation of]{Estimation of $\gamma_1$.} Consider the initial state $v(0)=\vin=(0,0,-1)$. Letting the system relax without control ($u=0$) for time $t$ leads to the evolution
$$
v(t)= 
R_t(\vin)
=
\begin{pmatrix}
0\\0\\1-2 e^{-\gamma_1 t}
\end{pmatrix}.
$$
For a fixed time $t_1$, we determine $p_1=h(v(t_1))$ to be $(1-z(t_1))/2=e^{-\gamma_1 t_1}$. Solving for $\gamma_1$, we recover 
$$
\gamma_1=-\frac{\log p_1}{t_1}.
$$
Following Section~\ref{S:estimation}, we let $\hat{p}_1=s_1/n$ be the estimator of $p_1$ by counting the number of excited state outcomes over $n$ repetition of the state. Then the estimator for $\gamma_1$ is given by
$$
\hat{\gamma}_1= - \frac{\log \hat p_1}{t_1}= - \frac{\log (s_1/n)}{t_1}.
$$
In particular $\E[\hat\gamma_1]=\gamma_1$ and (with the approximation $\Var (f(\hat p))$ by $f'(\E(\hat p))\Var(\hat p)f'(\E(\hat p))$)
$$
\Var(\hat\gamma_1)=\frac{e^{\gamma_1 t_1}-1}{n t_1^2}+o\left(\frac{1}{n}\right).
$$
\begin{remark}
For fixed $t_1$, the uncertainty depends on $\gamma_1$, and the estimation may be more or less reliable. The choice of $t_1$ can be left to a minimizaltion principle, that is, for $q\in[1,\infty]$, pick 
$$
t_1=\arg\min_{t_1>0} \left\|\gamma_1\mapsto\frac{e^{\gamma_1 t_1}-1}{n t_1^2}\right\|_{L^q([\gmin_1,\gmax_1])}.
$$
In the present paper, we will focus on the case  $q=\infty$, then $t_1$ is defined as the unique solution of \begin{equation}\label{E:deft1}
    e^{\gmax_1 t_1}(2-\gmax_1 t_1)=2.
\end{equation}
\end{remark}

\subsubsection[Estimation of]{Estimation of $\kappa$.} We realize a pulse with initial state $\vin$, of control $u=\umax$ of time $t_2=\tau_2/\umax$:
$$
v(t_2)
=
P^+_{\tau_2/\umax}(\vin).
$$
Now we push $\umax\to \infty$ in order to reach the asymptotic regime.
Setting $u=1/\varepsilon$, and $t_2=\varepsilon \tau_2$, this yields
$$
v(t_2)
=
e^{\kappa \Omega_x \tau_2 + \varepsilon A_0 \tau_2 } \vin
+
\int_0^{\varepsilon \tau_2} e^{A_u (\varepsilon \tau_2-s)} b \diff s.
$$
Let $\Omega(\varepsilon)= \kappa \Omega_x \tau_2 + \varepsilon A_0 \tau_2$ and change variables in the integral to $\sigma= s/(\tau_2\varepsilon)$  so that
$$
\int_0^{\varepsilon \tau_2} e^{A_u (\varepsilon \tau_2-s)} b \diff s
=
\tau_2 \varepsilon
\int_0^{1} e^{\Omega(\varepsilon) (1-\sigma)} b \diff \sigma.
$$
Letting $\varepsilon\to 0$, $\Omega(\varepsilon)\to \Omega(0)= \kappa \Omega_x \tau_2$ and
$$
v(\eps \tau_2)\xrightarrow[\eps\to 0]{} 
e^{\kappa \Omega_x \tau_2} \vin
=
\begin{pmatrix}
0\\ \sin \kappa t_2 \\ -\cos \kappa t_2
\end{pmatrix}
$$
Hence
$$
p_2=
h 
(\phi(\vin;1/\eps,\eps \tau_2))\xrightarrow[\eps\to 0]{} \frac{1}{2}(1+\cos \kappa \tau_2).
$$

We can now discuss the choice of $\tau_2$. Assuming first that $\kappa\tau_2<\pi$ ensures that we can compute the inversion
$$
\kappa
=
\frac{1}{\tau_2}\arccos
(2 p_2- 1).
$$
As before, following Section~\ref{S:estimation}, we let $\hat{p}_2=s_2/n$ be the estimator of $p_2$ (over $n$ repetition of the state). Then the estimator for $\kappa$ we have picked is
$$
\hat{\kappa}= \frac{1}{\tau_2}\arccos
(2 \hat p_2- 1),
$$
so that, in the idealized regime, $\E[\hat\kappa]=\kappa$ and 
\begin{equation}
\Var(\hat \kappa)=\frac{1}{n\tau_2^2}+o\left(\frac{1}{n}\right).    
\end{equation}
In order to minimize the variance, we are inclined to select
$$
\tau_2=\frac{\pi}{\kmax}.
$$
The above equation doesn't account for the approximation of the law of $\hat p_2$ as a Gaussian, which breaks in the case $\tau_2=\frac{\pi}{\kmax}$ and $\kappa=\kmax$ (which would theoretically put $p_2$ at 1). Therefore we pick, for a small arbitrary parameter $\beta\in(0,1)$ and the pulse time
$$
\tau_2=(1-\beta)\frac{\pi}{\kmax}.
$$
\begin{remark}
Since $p_2\in(0,1)$, the variance of $\hat{p}_2$ is bounded by $1/4 n$. 
This allows to bound below the necessary $n$ in order to preserve our approximation (in connection to $\kmin/\kmax$ and $\beta$.
Assuming we want $p_2$ to be $m\sigma$ away from $0$ or $1$, this that we must require $$\kappa \tau_2\in\left( \arccos \left(1-\frac{m}{\sqrt{n}}\right),\arccos \left(\frac{m}{\sqrt{n}}-1 \right)\right).$$
From $\arccos \left(x-1 \right)-\arccos \left(1-x\right)=2\arcsin 1-x$, this leads to
$$
\tau_2(\kmax-\kmin)<2\arcsin \left(1-\frac{m}{\sqrt{n}}\right).
$$
If $\tau_2=(1-\beta)\frac{\pi}{\kmax}$, for what $n$ is the above satisfied?
Simplifying $(1-\beta)\pi(1-\frac{\kmin}{\kmax})<2\arcsin \left(1-\frac{m}{\sqrt{n}}\right)$ yields
$$
n>\frac{m^2}{\left(1-\sin\left( (1-\beta)(1-\frac{\kmin}{\kmax})\frac{\pi}{2} \right)\right)^2}.
$$
For example, picking $m=5$, $\beta=1/10$, $\kmin/\kmax=1/10$ yields approximately $n>1.28\times 10^4$.
\end{remark}

\subsection{Parameters orthogonal to the observation}

As discussed in Section~\ref{S:estimation}, in order to access the parameters $\omega,\gamma_2$, we follow a framework pulse--relaxation--opposite-pulse. Recalling $h(v)=\frac{1}{2}(1-z)$, we select the two experimental observables $p_3,p_4$ to be 
$$
\begin{aligned}
p_3 &=h(P_{\eps \tau_2}^{-}\circ R_{t_3} \circ P_{\eps \tau_2}^+(\vin)),
\\
p_4&=h(P_{\eps \tau_2}^{-}\circ R_{2 t_3} \circ P_{\eps \tau_2}^+(\vin)).
\end{aligned}
$$
With this last element, we are now able to give an expression for the mapping $F_\eps$:
\begin{equation}
F_\eps(\theta)
=
\begin{pmatrix}
h(R_{t_1}(\vin))
\\
h(P_{\eps \tau_2}^+(\vin))
\\
h(P_{\eps \tau_2}^{-}\circ R_{t_3} \circ P_{\eps \tau_2}^+(\vin))
\\
h(P_{\eps \tau_2}^{-}\circ R_{2 t_3} \circ P_{\eps \tau_2}^+(\vin))
\end{pmatrix}
\end{equation}

Since $h$ and $v\mapsto\phi(v ,u,t)$ are affine maps, the expression for $F_0$ is easily by passing $P_{\eps \tau_2}^\pm$ to its limit, the $x$-rotation $v\mapsto e^{\pm\kappa \Omega_x \tau_2}$.

Then we prove the following lemma, which  we announced in Section~\ref{S:quantification}.
\begin{lemma}\label{L:diffeo}
Up to a choice of $t_1,\tau_2,t_3$, the map $F_0:\Theta\to (0,1)^4$ is a diffeomorphism.
\end{lemma}

\begin{proof}
The durations $t_1$, $\tau_2$ have already been chosen. Let us see what can be said about $p_3,p_4$.
An explicit expression of $R_{t}$ is computable:
$$
e^{A_0 t}
=
\begin{pmatrix}
 e^{-\frac{1}{2} (\gamma_1+4 \gamma_2)t} \cos  \omega t & -e^{-\frac{1}{2}  (\gamma_1+4 \gamma_2)t} \sin \omega t & 0 \\
 e^{-\frac{1}{2}  (\gamma_1+4 \gamma_2)t} \sin \omega t & e^{-\frac{1}{2}  (\gamma_1+4 \gamma_2)t} \cos \omega t & 0 \\
 0 & 0 & e^{-\gamma_1 t}
\end{pmatrix}
$$
and 
$$
\int_0^t e^{A_0 (t-s)} b \diff s=
\begin{pmatrix}
0 
\\
0 
\\
1-e^{-\gamma_1 t}
\end{pmatrix}.
$$
From this we deduce
\begin{multline}\label{E:fullh}
h(P_{\eps \tau_2}^{-}\circ R_{t} \circ P_{\eps \tau_2}^+(\vin))
=
\frac{1}{2} \Big(
1
-\cos (\kappa \tau_2) \left(1-e^{-\gamma_1 t}-\cos \kappa \tau_2\right)
\\
+\sin ^2(\kappa \tau_2) e^{-\frac{1}{2}(\gamma_1+4\gamma_2)t} \cos  \omega t \Big)    
\end{multline}
Since $\gamma_1$ and $\kappa$ are fully determined by the first two lines of $F_0$, we only have to show that for fixed $(\gamma_1,\kappa)\in [\gmin_1,\gmax_1]\times[\kmin,\kmax]$, the mapping 
$$
(\gamma_2,\omega)\mapsto
\begin{pmatrix}
h(P_{\eps \tau_2}^{-}\circ R_{t_3} \circ P_{\eps \tau_2}^+(\vin))
\\
h(P_{\eps \tau_2}^{-}\circ R_{2 t_3} \circ P_{\eps \tau_2}^+(\vin))
\end{pmatrix}
$$
is a diffeomorphism. From expression~\eqref{E:fullh}, we can see that this reduces to showing that 
\begin{equation}\label{E:defpsi}
\psi:(\gamma_2,\omega)\mapsto
\begin{pmatrix}
e^{-\frac{1}{2}(\gamma_1+4\gamma_2)t_3} \cos  \omega t_3 
\\
e^{-(\gamma_1+4\gamma_2)t_3} \cos  2\omega t_3
\end{pmatrix}
=:
\begin{pmatrix}
q_3\\q_4
\end{pmatrix}
\end{equation}
is a diffeomorphism onto its image (by construction of $\tau_2$, $\sin(\kappa \tau_2)\neq0$). The easiest is to prove it via inverse function theorem. If $\omax t_3<\pi$, then for all $(\omega,\gamma_2)\in[\omin,\omax]\times[\gmin_2,\gmax_2]$
$$
\det \Jac \psi =
-4 t_3^2 e^{-\frac{3}{2}  (\gamma_1+4 \gamma_2)t_3} \sin  \omega t_3<0.
$$
From this we deduce that $\psi$ is a diffeomorphism on its domain, which then extends to $F_0$ and concludes the proof. 
\end{proof}

\begin{remark}
As appears in the proof, an ideal choice for the pulse would be to pick rather than $\tau_2$, an adaptive time $\tau_2'=\pi/(2 \hat{\kappa})$ in order to maximize  $\sin(\kappa \tau'_2)^2$, and, in turn, the role of $\gamma_2$ and $\omega$ in $F_0$. This, however, complicates the interconnection between the bias and the estimation, but would be an interesting question to investigate.
\end{remark}

Let us now discuss the inversion of $F_0$. 
Let first us recover $\gamma_2$ and $\omega$ as functions of $(q_3,q_4)$.
Since $\cos  2\omega t_3 = 2 \cos^2(\omega t_3)-1$, we isolate $\gamma_2$ as
$$
\begin{aligned}
\cos(\omega t_3) = q_3  e^{\frac{1}{2}(\gamma_1+4\gamma_2)t_3} ,
\qquad \qquad 
q_4 &= e^{-(\gamma_1+4\gamma_2)t_3}(2q_3^2  e^{(\gamma_1+4\gamma_2)t_3}-1)     
\\
&=
2q_3^2-e^{-(\gamma_1+4\gamma_2)t_3}
\end{aligned}
$$
As a result
$$
e^{-(\gamma_1+4\gamma_2)t_3}=2q_3^2-q_4
\;\text{ and }\;
\cos \omega t_3 
= 
\frac{q_3}{\sqrt{2q_3^2-q_4}}.
$$
This allows to separate $\gamma_2$ and $\omega$.
Next, we can express $q_3,q_4$ in terms of $p$ using the relation
$$
\begin{aligned}
2p_2(1-p_2)q_3&=p_3+p_2-1+p_2(1-2p_2)p_1^{t_3/t_1},
\\
2p_2(1-p_2)q_4&=p_4+p_2-1+p_2(1-2p_2)p_1^{2t_3/t_1}.
\end{aligned}
$$
Hence we effectively get the inversion $p\mapsto (\gamma_2,\omega)$ through
$$
\gamma_2=-\frac{1}{4 t_3}\log\left( \frac{2q_3^2-q_4}{p_1^{t_3/t_1} }\right),\qquad
\omega =\frac{1}{t_3}\arccos \left(\frac{q_3}{\sqrt{2q_3^2-q_4}}\right).
$$

The choice of $t_3$ is linked to multiple competing requirements.  
In particular, the inversion problems associated with an exponential and a cosine factor have different characters.  To obtain a useful estimator, $t_3$ should be large enough that the exponential factor is not too close to $1$, while simultaneously having $\omega t_3$ lying well inside $(0,\pi)$. If $\omega$ is multiple orders of magnitude larger than $\gamma_2$, both cannot be satisfied simultaneously. The ambiguity in $\omega$ can be handled by the local adaptive technique discussed in Section~\ref{S:local}, allowing to focus on a reliable recovery of $\cos(\omega t)$ for a suitable set of times $t$, while $t_3$ may be chosen to prioritize estimation of $\gamma_2$.

However, for the present discussion and the establishment of Theorem~\ref{T:mainthm} we focus on a conservative and non-adaptive strategy that guarantees that $F_0$ is a global diffeomorphism on $\Theta$. Concretely, we pick $t_3$ so that the interval $[\omega_{\min}t_3,\omega_{\max}t_3]$ is centered inside $(0,\pi)$, which leads to the simple choice
$$
t_3=\frac{\pi}{\omega_{\min}+\omega_{\max}}.
$$
This choice simplifies the injectivity discussion, however if the domain $[\omega_{\min},\omega_{\max}]$ allows, it may be better to pick a value such that $[\omega_{\min}t_3,\omega_{\max}t_3]$ is centered inside $(k\pi ,(k+1)\pi)$ for some integer $k$.
As we illustrate in the numerical section, the conservative choice $k=0$ degrades the uncertainty in $\gamma_2$ when $\gamma_2\ll\omega$. In practice, favoring $\gamma_2$ (by picking $t_3$ to optimize sensitivity in $\gamma_2$ as we did for $\gamma_1$) and disambiguing $\omega$ can mitigate this issue (for instance allowing to pick $k\neq 0$).

\begin{remark}
In essence, the inversion of $F_0$ above separates $\gamma_2$ and $\omega$ by introducing two new virtual observables $\left(2q_3^2-q_4,{q_3}/{\sqrt{2q_3^2-q_4}}\right)$, each depending on a different parameter. So it would be possible to add flexibility to the overall method by repeating the same construction with a time $t_4$ in order to obtain new observables $(p_5,p_6)$, and associated $(q_5,q_6)$. Then this allows to tune $t_3$ for the estimation of $\gamma_2$, and $t_4$ for $\omega$. This, however, requires defining $F_0$ not as a diffeomorphism but as an embedding, and then designing a pseudo-inverse, which goes slightly beyond the scope of the present paper.
\end{remark}

\section{Finite-amplitude pulses: bias analysis and bounds}\label{S:bound}
As discussed in Sections~\ref{S:estimation}-\ref{S:quantification}, falling outside of the idealized asymptotic regime of infinite pulses of zero duration introduces a bias in the estimation that we estimate here.

We will need the following notation:
letting
$$
\lambda=\min(\gamma_1,\frac{1}{2}(\gamma_1+4\gamma_2)),
\qquad
\mu = \max
\left(
\gamma_1,
\left(
(\tfrac{\gamma_1}{2}+\gamma_2)^2+\omega^2
\right)^{1/2}
\right),
$$
we set 
$$
C(\theta)=\frac{\tau_2}{2}(\mu+ \gamma_1),
\qquad
C'(\theta,t)=C(\theta)+
\frac{\tau_2}{2} \mu
\big(
2 e^{-\lambda t}-e^{-\gamma_1 t}
\big)>0.
$$
Using these functions, we obtain the following bounds, which effectively proves Proposition~\ref{P:boundInit}.
\begin{proposition}\label{P:bound_effective}
With $F_\eps=(F_\eps^j)_{1\leq j\leq 4}$, we have the bounds:
$$
|F_\eps^1(\theta)-F_0^1(\theta)|=0,
$$
$$
|F_\eps^2(\theta)-F_0^2(\theta)|\leq \eps C(\theta),
$$
$$
|F_\eps^3(\theta)-F_0^3(\theta)|\leq \eps C'(\theta,t_3),
$$
$$
|F_\eps^3(\theta)-F_0^3(\theta)|\leq \eps C'(\theta,2t_3).
$$
\end{proposition}

\begin{proof}
We propose a strategy based on the mean value theorem. That is, letting $\varepsilon=1/\umax$ be fixed, we consider the dynamics of a pulse under control $u=\pm1/\eta$, duration $\eta\tau_2$ with $\eta\in [0,\eps]$.
Let us introduce some usefull notations for the present section.
For $\eta\geq 0$, let $\Omega^\pm(\eta)= \pm\kappa \Omega_x \tau_2 + \eta A_0 \tau_2$ and let
$$
M_\eta^\pm = e^{\Omega^\pm(\eta)},\qquad V_\eta^\pm = \eta\tau_2\int_0^1 e^{\Omega^\pm(\eta)(1-\sigma)}b\diff \sigma.
$$
Then for any $v\in \R^3$
$$
P^\pm_\eta(v)= M_\eta^\pm v + V_\eta^\pm.
$$
We also denote by $W_t=\int_0^t e^{A_0 (t-s)} b \diff s$, so that $R_t(v)=e^{A_0t}v+W_t$.
From this we determine
$$
\begin{aligned}
P_\eta^{-}\circ R_{t} \circ P_\eta^+(v)
&=
M_\eta^-\left(e^{A_0 t}(M_\eta^+ v +V_\eta^+)+W_t\right)+V_\eta^-
\\
&= 
M_\eta^- e^{A_0 t}M_\eta^+ v +  M_\eta^- e^{A_0 t}V_\eta^+ 
 + M_\eta^- W_t +V_\eta^-.
\end{aligned}
$$
{
Then (with $V_0^\pm=0$)
$$
\begin{aligned}
P_\eta^{-}\circ R_{t} \circ P_\eta^+(v)-P_0^{-}\circ R_{t} \circ P_0^+(v)
=
&(M_\eta^- -M_0^-) e^{A_0 t}M_\eta^+ v 
+
M_0^- e^{A_0 t}(M_\eta^+-M_0^+) v 
\\&+  
M_\eta^- e^{A_0 t}V_\eta^+ 
+
(M_\eta^- -M_0^-) W_t 
+
V_\eta^-.
\end{aligned}
$$
We now bound each element in this sum independently by using the mean value theorem.
For all $s\in \R$, we have the formula (see, e.g., \cite{NAJFELD1995321})
$$
\frac{\partial M_\eta^\pm}{\partial \eta}
=
\tau_2
\int_0^1
e^{\Omega^\pm(\eta)(1-\tau)}
A_0
e^{\Omega^\pm(\eta)\tau}\diff \tau.
$$
Notice that $\Omega(\eta)$ is contractive in Euclidean norm:
if $\dot w  = \Omega(\eta) w$, then
$$
\frac{\diff |w|^2}{\diff t}
=
w^\top (\Omega(\eta)+\Omega(\eta)^\top) w
\leq 
-\eta \min(2\gamma_1,\gamma_1+4\gamma_2)|w|^2.
$$
Thus, in operator 2-norm, 
$$
\|M_\eta^\pm\|\leq 1, \quad 
\left\|
\frac{\partial M_\eta^\pm}{\partial \eta}
\right\|
\leq 
\tau_2\|A_0\|.
$$
Hence, in particular, with $\mu = \|A_0\|$,
$\|M_\eta^\pm-M_0^\pm\|\leq \eta \tau_2 \mu$.
On the other hand, the contraction property yields $|V^\pm_\eps|\leq \eta \tau_2 |b|=\eta \tau_2 \gamma_1$. In particular, we immediately get
$$
|P_\eta^+(\vin)-P_0^+(\vin)|\leq \eta \tau_2(\mu+\gamma_1)=2\eta C(\theta).
$$
Now we have 
$$
\|e^{A_0t}\|=e^{-\lambda t},
\quad 
|W_t|= 1-e^{-\gamma_1 t}.
$$
so that 
\begin{multline*}
\big|P_\eta^{-}\circ R_{t} \circ P_\eta^+(\vin)-P_0^{-}\circ R_{t} \circ P_0^+(\vin)\big|
\leq
\\
\eta \tau_2\left(2 \mu e^{-\lambda t}+\gamma_1 e^{-\lambda t}+\mu(1-e^{-\gamma_1 t})+\gamma_1\right)
=2\eta C'(\theta,t)
\end{multline*}
In conclusion, using the fact that $\Jac h=(0,0,1/2)$, we obtain the statement at $\eta=\eps$.

}

\end{proof}

As explained in Section~\ref{S:quantification}, this implies that $F_\eps$ is a diffeomorphism when $\eps$ is small enough. In order to prove Theorem~\ref{T:mainthm}, what remains to be proved is the existence (and description) of the box $\Delta$.

Differentiating the identity $F_\eps\circ F_\eps^{-1}(p)=p$
with respect to $\varepsilon$ yields (with $F_\eps(\theta) = p$)
\begin{equation*}
\label{eq:diffeq-inverse}
\frac{\partial F_\varepsilon^{-1}}{\partial \eps}(p)
= -\left(\Jac F_\varepsilon(\theta)\right)^{-1}
\frac{\partial F_\varepsilon(\theta)}{\partial \eps} .
\end{equation*}

Letting $H_\varepsilon=(F_\eps(\theta)-F_0(\theta))/\eps$
so that
$$
\Jac F_\varepsilon(\theta)^{-1}
=
\left(\id +\eps\Jac H_\varepsilon(\theta)\Jac F_0(\theta)^{-1}\right)^{-1}\Jac F_0(\theta)^{-1}.
$$
Under the condition 
$$
\eps
\sup_{\theta\in\Theta}
\left\| \Jac H_\varepsilon(\theta)\Jac F_0(\theta)^{-1}\right\|
< 1
$$
we obtain the uniform bound
\begin{equation*}
\label{eq:inverse-bound}
\sup_{\theta\in\Theta}
\left|(\Jac F_\eps(\theta))^{-1}\right|
\leq
\frac{\left\|\Jac F_0(\theta)^{-1}\right\|}{1-\eps \|\Jac H_\eps(\theta)\Jac F_0(\theta)^{-1}\|}
\end{equation*}
This proves the existence of $G_\eps$ as defined in \eqref{E:defG} and implies invertibility of $F_\eps$, as stated in Section~\ref{S:quantification}. 

The above fact alone proves the existence of a $\Delta$ box, but we can compute one such box as follows. We use the approximation
$$
\frac{\partial F_\varepsilon^{-1}}{\partial \eps}(p)
= -\left(\Jac F_0(\theta)\right)^{-1}
\left.\frac{\partial F_\varepsilon(\theta)}{\partial \eps}\right|_{\eps=0} +O(\eps),
$$
in conjunction with Proposition~\ref{P:bound_effective}. Let 
$$
\Delta'(\theta)=
\{0\}\times [-C(\theta),C(\theta)]\times 
[-C'(\theta,t_3),C'(\theta,t_3)]
\times [-C'(\theta,2t_3),C'(\theta,2t_3)].
$$
and let $\pi_i:\R^4\to \R$ be the projection on the $i$-th coordinate, $1\leq i\leq 4$. We then set for $1\leq i\leq 4$ 
$$
\Delta_i=\sup_{\theta\in \Theta }
\pi_i\left[\left(\Jac F_0(\theta)\right)^{-1}
\cdot \Delta'(\theta)\right].
$$
Such a set $\Delta$ then matches the definition and Theorem~\ref{T:mainthm} follows from the analysis in Section~\ref{S:quantification}.

\section{Adaptive non-global inversion}
\label{S:local}

In Section~\ref{S:ideal} we established that, for a suitable choice of pulse and relaxation times, the ideal map $F_0$ (and therefore $F_\varepsilon$ for $\varepsilon$ sufficiently small) is a  diffeomorphism. These times are subject to strong constraints, so it may be preferable to adopt a more adaptive strategy that relies on local rather than global inversion.
Indeed, for generic choices of durations, both $F_0$ and $F_\varepsilon$ are local diffeomorphisms on an open and dense subset of $\Theta$. Loss of invertibility only occurs at the critical configurations where 
$$
\cos(\kappa \tau_2)=\pm 1, \qquad \cos(\omega t_3)=\pm 1.
$$
For generic pairs of durations $(\tau_2,\tau_2')$ and $(t_3,t_3')$, at every point of $\Theta$ at least one combination yields a locally invertible Jacobian. The drawback, however, is that the corresponding domain of invertibility may be arbitrarily small, making a purely local approach still unreliable.

Nevertheless, for any fixed choice of positive durations, the preimage $F_0^{-1}(\{p\})$ of a point $p\in (0,1)^4$ consists of at most a  finite set of points in $\Theta$. Hence the main practical difficulty is to identify the correct preimage among finitely many admissible candidates.

The source of this ambiguity lies in the contributions of $\kappa$ and $\omega$, whose identification reduces to solving an equation of the form
$$
\cos(\zeta t)=f(p(\zeta,t)),
$$
where $\zeta$ denotes either parameter. For a given measurement at time $t$, let $\{\zeta_i\}$ denote the finite set of admissible solutions (in restriction to a prescribed compact interval). That is, $\cos(\zeta_i t)=f(p(\zeta,t))$ for elements $\zeta_i$.
To isolate the true value, we can perform a second set of measurements at a different time~$t'$. 
If $t'$ is chosen to be rationally independent of $t$ (i.e. $t/t'\neq \mathbb{Q}$), then $\cos \zeta_i t'\neq \cos \zeta_j t'$ if $i\neq j$,
so the second experiment distinguishes the true parameter value from all other candidates.

In practice, the probabilities $p_i$ are only estimated and the pulses are not ideal. The practical implementation is therefore as follows. Fix an admissible tolerance $\varepsilon_0>0$. Perform two identification procedures with the same relaxation time $t_1$ but using two different pairs of pulse durations, $(\tau_2,t_3)$ and $(\tau_2',t_3')$, where $\tau_2'$ and $t_3'$ are randomly chosen. For each procedure we compute the corresponding sets of admissible parameter estimates, and we retain all pairs of candidates whose images under $F_0$ lie within $\varepsilon_0$ of the experimentally measured probabilities. 
If several candidates remain, we repeat the procedure with an additional randomly chosen pair $(\tau_2'',t_3'')$, and so on.

This approach can be also be used in conjunction with the global inversion approach by reducing the size  of the a priori set $\Theta$. However if the a priori parameter bounds are large, the above filtering may be slow to converge.

\section{Numerical experiments}
\label{S:simulations}

In this section, we illustrate the performance of the proposed estimation protocol through numerical simulations. The dynamics of the open qubit (Eq.~\ref{E:dynamics}) were integrated using the Julia library \texttt{DifferentialEquations.jl}, utilizing a high-order {\texttt{Tsit5}} integrator to ensure that numerical errors are negligible compared to statistical fluctuations.

To illustrate the estimator, we adopt a  parameter set adapted from~\cite{morzhin2019minimal}. While physical implementations (e.g., superconducting qubits) exhibit frequency-to-dephasing ratios of order $10^6$, simulating such dynamics hides the algorithmic properties. We therefore utilize a academic numerical values that preserve the essential physical hierarchy $\omega \gg \kappa \gg \gamma$ while ensuring numerical stability. Precise values for $\theta, \Theta$, experimental times, as well as averaged  estimator performance (RMSE), are summarized in Table~\ref{tab:results}. These results are obtained under a strong control regime  $u_{\max} = 10^5$ and large sample size $n = 5 \cdot 10^8$.

\begin{table}[ht]

\caption{Simulation parameters and Root Mean Square Error calculated over 100 trials.}

\centering
\label{tab:results}
    \begin{tabular}{l c c c c} 
    \hline
    \textbf{Param} & \textbf{True Value} $\theta^*$ & \textbf{Bounds} $[\underline{\theta}, \bar{\theta}]$ & \textbf{Times} & \textbf{RMSE} ($10^{-4}$) \\
    \hline
    $\omega$   & $2$ & $[1, 4]$ & $t_3 = 0.62$ & \textbf{3.58} \\
    $\kappa$   & $0.015$ & $[0.01, 0.04]$ & $\tau_2 = 62.83$ &\textbf{2.86} \\
    $\gamma_1$ & $0.002$ & $[0.001, 0.003]$ & $t_1 = 530.0$ & \textbf{0.015} \\
    $\gamma_2$ & $0.003$ & $[0.002, 0.005]$ & $t_3 = 0.62$ & \textbf{1.12} \\
    \hline
    \end{tabular}%
\end{table}

Next, we illustrate the asymptotic behavior of the estimator $\hat{\theta}$ as a function of the measurement budget $n$. Theorem~\ref{T:mainthm} guarantees $1/\sqrt{n}$ convergence for infinite pulses, which becomes only practical under finite pulse. Proposition 4.1 predicts a systematic bias of order $O(1/\umax)$. To illustrate this, we fixed $\umax$ and varied the sample size $n$, see Figure~\ref{F:curve}. We can observe that the order of convergence of the estimator tapers off as $n$ get large.
Finally, we further illustrate Theorem~\ref{T:mainthm} by tracing the (first order) confidence regions $\mathrm{CR}_{1-\alpha}$, $\alpha=0.01$, for two different values of $u_{max}$, and $n=10^{9}$. See Figure~\ref{F:regions}. 

\begin{figure}[ht]
    \centering
    \includegraphics[width=.7\linewidth]{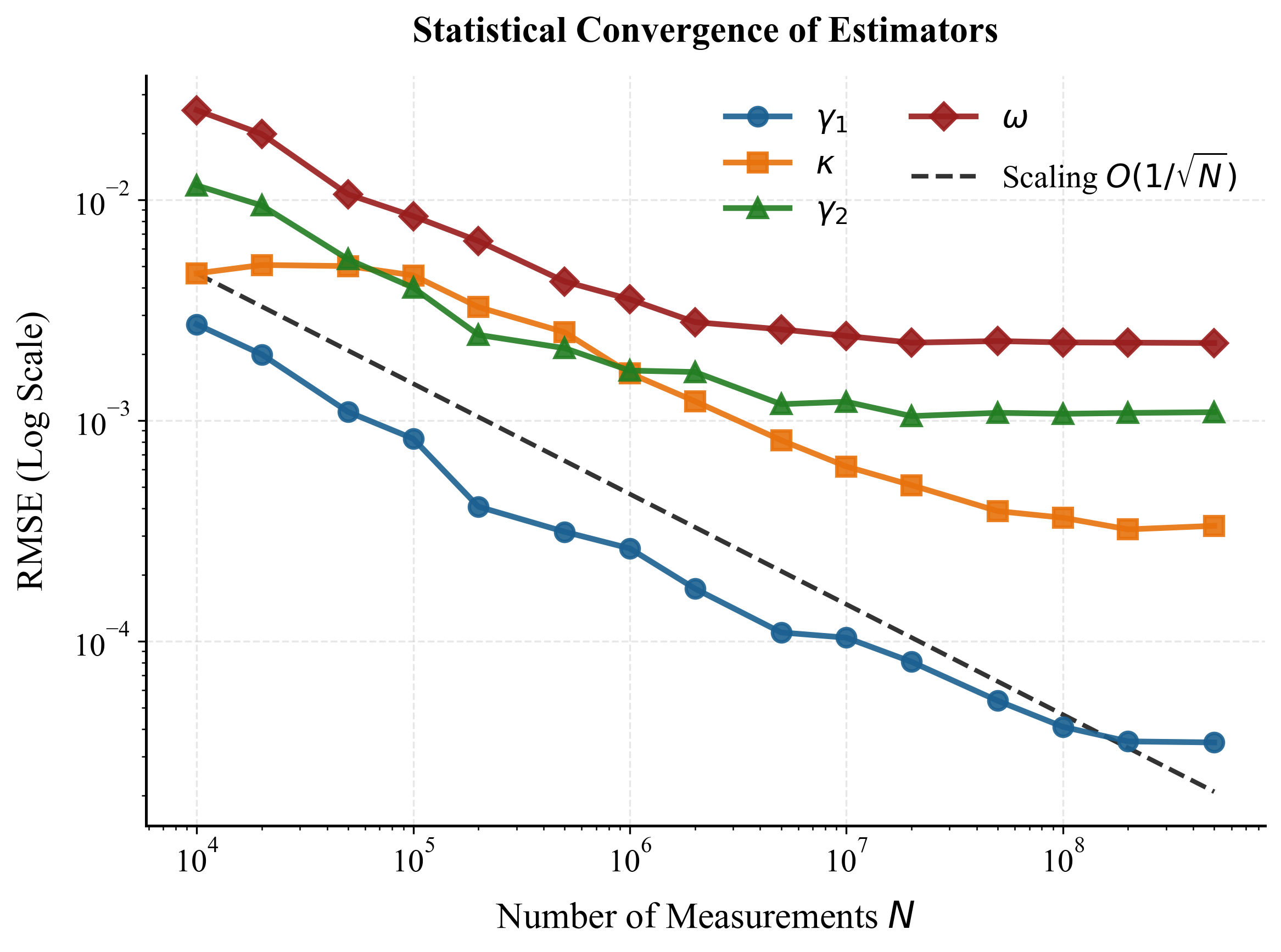}
    \begin{minipage}{.8\linewidth}
        \caption{Convergence of the estimators for $\omega, \kappa, \gamma_1, \gamma_2$ vs. number of measurements $n$.}
    \end{minipage}
    
    \label{F:curve}
\end{figure}

\begin{figure}[ht]
    \centering
    \includegraphics[width=.5\linewidth]{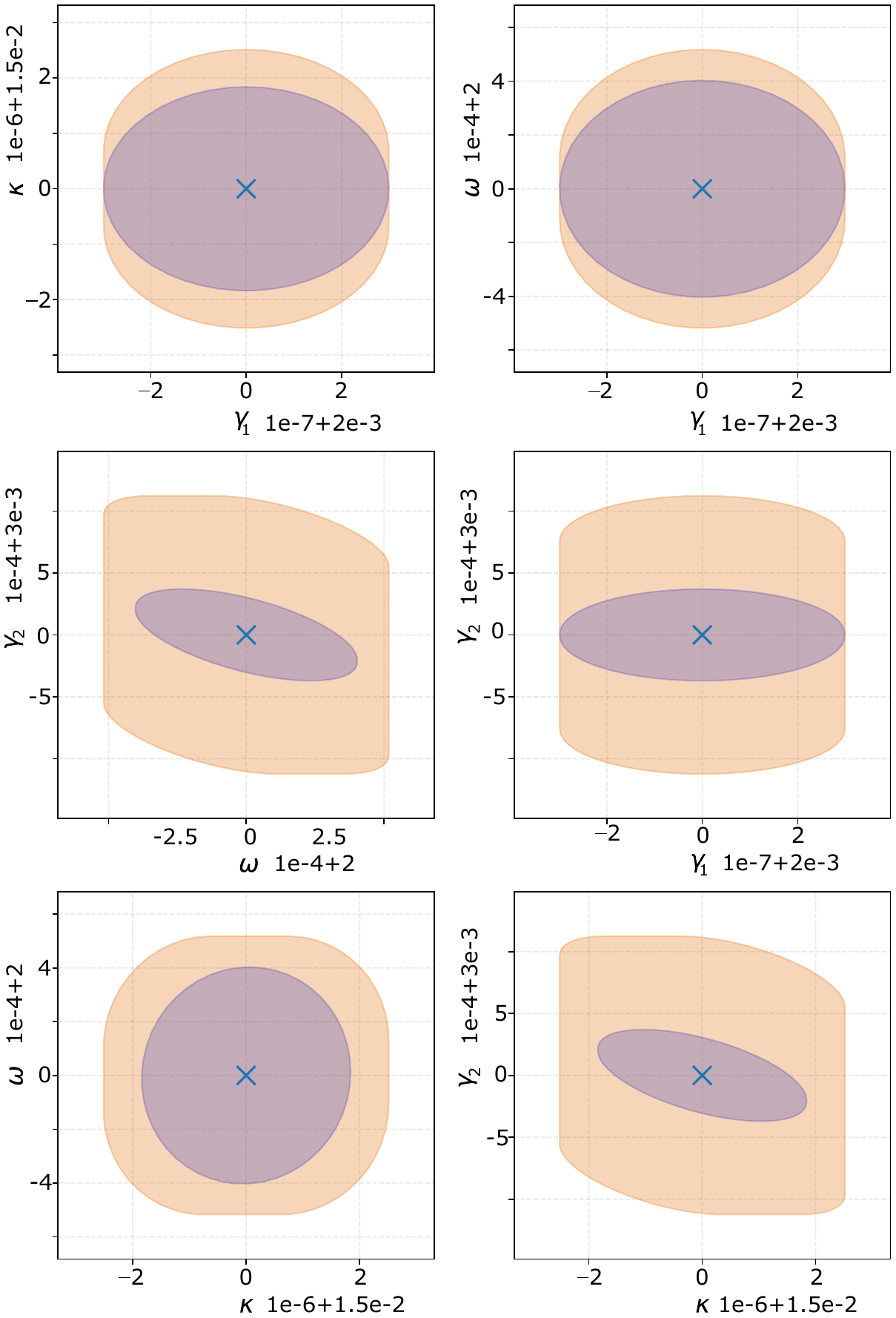}
    
    \begin{minipage}{.8\linewidth}
    \caption{First order confidence regions for $\umax = 10^5$ (orange), $\umax = 10^7$ (purple) and fixed $n=10^9$. Relative center and scale are indicated for each graph.\label{F:regions}}
    \end{minipage}
\end{figure}

\section{Conclusion}
In this work, we introduced a systematic and interpretable methodology for identifying the parameters of an open qubit system subject to relaxation and dephasing. By operating in the idealized regime of infinite-amplitude pulses, we demonstrated that the system parameters can be analytically reconstructed from a minimal set of experimental observables. We further analyzed the impact of finite-amplitude pulses, quantifying their perturbative effect on the estimation protocol and providing a rigorous framework to separate statistical uncertainty from modeling errors. This allowed us to construct a confidence region that explicitly accounts for both sources of uncertainty.
While the current approach prioritizes theoretical guarantees through diffeomorphic mappings, we highlighted several promising avenues for future work. These include relaxing some of the conservative assumptions, such as adopting embedding-based strategies instead of strict diffeomorphisms, and integrating the adaptive disambiguation algorithm more tightly with the uncertainty analysis.

\bibliographystyle{abbrv}
\bibliography{biblio}

\end{document}